\documentclass[a4paper,10pt]{article}
\usepackage[pdftex]{graphicx}
\usepackage{graphicx}
\usepackage{subfigure}
\usepackage{ amsfonts, amsthm,  mathtools}

\newtheorem{thm}{Theorem}[section]

\newtheorem{cor}{Corollary}
\newtheorem{guess}{Hypothesis}[section]
\theoremstyle{definition}
\newtheorem{Def}[thm]{Definition}
\newtheorem{exmp}{Example}
\theoremstyle{remark}
\newtheorem{rem}{Remark}[section]

\numberwithin{equation}{section}
\newcommand{\ep}{\epsilon}

\providecommand{\norm}[1]{\lVert#1\rVert}
\providecommand{\abs}[1]{\lvert#1\rvert}
\providecommand{\Abs}[1]{\left\lvert#1\right\rvert}
\newcommand {\bee}{\begin{equation*}}
\newcommand {\ee}{\end{equation*}}
\newcommand {\beaa}{\begin{eqnarray*}}
\newcommand {\eaa}{\end{eqnarray*}}
\newcommand {\al} {\alpha}
\newcommand {\la} {\lambda}
\newcommand {\sig} {\sigma}

\newcommand {\fract}[2]{\mbox{${\displaystyle{\frac{#1}{#2}}}$}}
\newcommand {\m}[1]{$\mbox{tr}(#1)$}
\newcommand {\mm}[2]{$\mbox{tr}_{#1}(#2)$}
\newcommand {\be}{\begin{equation}}
\newcommand {\e}{\end{equation}}
\newcommand {\bea}{\begin{eqnarray}}
\newcommand {\ea}{\end{eqnarray}}
\newcommand {\R}{\mathbb{R}}

\newcommand {\C}{\mathbb{C}}
\title{Numerical evaluation of operator determinants}
\author{Issa Karambal}
\begin{document}
\date{}
\maketitle
\begin{abstract}
For any integral operator $K$ in the Schatten--von Neumann classes of compact operators and its approximated operator $K_N~(N\ge1)$ obtained by using for example a quadrature or projection method, we show that the convergence of the approximate $p$-modified Fredholm determinants $\sideset{}{_{Np}}\det(I_N+zK_N)$ to the $p$-modified Fredholm determinants $\sideset{}{_p}\det(I_\mathcal{H}+zK)$  is uniform for all $p\ge1$. As a result, we give the rate of convergences when evaluating at an eigenvalue or at an element of the resolvent set of $K$.
\end{abstract}

\section{Introduction}
Let $K$ be an integral operator defined on a Hilbert space $\mathcal{H}$ and given by
\bee
  Ku(x)=\int_Xk(x,y)u(y)\mathrm{d}y,\quad x\in X 
\ee
where $k(x,y)$ is such that for all $x\in X$,
\be\label{eq:intro2}
k(x,.)u(.)\in L^1(X)\quad\text{and}\quad Ku\in\mathcal{H}. 
\e
Consider the following integral equation
\be\label{eq:intro1}
( I_\mathcal{H}+zK )u=f.
\e
In this paper, we are not considering solving the above integral equation but rather the eigenvalue problem that is, when $f=0$ for all $x\in X$. More precisely, we are interested in locating the eigenvalues of the operator $K$. For such a purpose, the approach used here is the numerical evaluation of the $p$-modified Fredholm determinants that is, $\sideset{}{_p}\det(I_\mathcal{H}+zK)$ for $p\ge1$ \cite{BS05,ISN00}. Indeed, the motivation behind this choice, comes from the fact that the the $p$-modified Fredholm determinants are entire functions whose reciprocal zeros are the eigenvalues of the operator $K$ with algebraic multiplicties equal to the order of the zeros \cite{FS65,ISN00,BS05,AG56,HB69}. The first appearance of these functions, in particular the $2$-modified Fredholm determinants goes back to Hilbert \cite{DH04} followed by Calerman~\cite{TC21}, Smithies \cite{FS65}, Pelmelj~\cite{JP04} and others. Hilbert observed that in order to achieve a convergence of the determinant associated with an $L^2$-kernel, it suffices to set $k(x,x)$ to zero in the formula of the Fredholm determinant. On the other hand, the Fredholm determinant, as introduced by Fredholm \cite{Fred}, is an entire function of the spectral parameter which characterise the solvability of equation \eqref{eq:intro1}, under the assumption that the nonzero function $f$ and the kernel $k$ are continuous. Broadly speaking, the extension of Hilbert's theory to kernels which are not square-integrable has been developed by Gohberg and Krein \cite{GK69} and Dunford and Schwartz \cite{DS63}. A particular attention for $4$-modified Fredholm determinant was conducted by Brascamp \cite{HB69}. Regarding the convergence issue, this also goes back to Hilbert. Indeed by applying the rectangular rule in equation \eqref{eq:intro1}, Hilbert showed  that, for a continuous kernel function $k$, the discrete version of the Fredholm determinant converges uniformly to the continuous one. Bornemann, in \cite{FB10}, extended Hilbert's result to any quadrature rule which converges for continuous functions. However in this paper, we prove the uniform convergence under weaker assumption of the kernel. Precisely, for kernel function $k$ satisfying condition \eqref{eq:intro2} such that the corresponding integral operator $K$ is in the Schatten--von Neumann classes of compact operators. Thus the desired convergence is obtained under the hypothesis that the set $\{K_n\}$ of numerical integral operators approximating $K$ is collectively compact \cite{ANS71}. This is because if the approximation error $(Ku-K_nu)$ is uniform then the sequence of eigenvalues $\{\la_n\}$ of $K_n$ converges uniformly to an eigenvalue $\la$ of $K$ \cite{ANS71,KA67,KA75,GV67}. The paper is organised as follows; in Section 2 we briefly review some properties concerning the $p$-modified Fredholm determinants. In Section 3, we show that the $p$-modified finite dimensional determinants associated with the operator $K_n$ converges uniformly to the $p$-modified Fredholm determinants. As a consequence, we give the rate of convergence when evaluating at an eigenvalue or at an element of the resolvent set of $K$. Finally in Section 4, we present some numerical results that demonstrate the analysis carried out in Section 3.
\section{The $p$-modified Fredholm determinants}\label{sec:1}
We briefly recall some basics that we use in this paper. Let $\mathcal{H}$ denote a separable Hilbert space with an inner product $(.~,~.)$, and $\mathcal{C}(\mathcal{H})$ the set of compact operators acting on $\mathcal{H}$. Then the Schatten--von Neumann classes of compact operators is given by
\bee 
\mathcal{J}_p(\mathcal{H})=\{K\in\mathcal{C}(\mathcal{H}): \norm{K}^p_p<\infty,~\forall p\ge 1\},
\ee
where $\norm{K}^p_p=\mbox{tr}(\abs{K}^p)$ with $\lvert K\rvert=(K^*K)^{1/2}$. Let $A\in\mathcal{J}_p(\mathcal{H})$, then for all $z\in\C$, the Pelmelj--Smithies formula for the $p$-modified Fredholm determinants is given by \cite{ISN00,BS05}
\be \label{eq1:sec1}
\sideset{}{_p}\det(1+z K)=\sum^\infty_{n=0}z^n\al^{(p)}_n(K),
\e
where 
\be \label{eq2:sec1}
\al^{(p)}_n(K)=\fract{1}{n!}\sideset{}{_{\C^n}}\det
\begin{pmatrix}
 \nu^{(p)}_1&n-1&0&\cdots&0\\
\vdots&\vdots&\vdots&\vdots&\vdots\\
\nu^{(p)}_{n-1}&\nu^{(p)}_{n-2}&\cdots&\cdots&1\\
\nu^{(p)}_{n}&\nu^{(p)}_{n-1}&\cdots&\cdots&\nu^{(p)}_1
\end{pmatrix}
\e
with
\be \label{eq3:sec1}
\nu^{(p)}_j=
\begin{cases}
 \mbox{tr}(K^j)&j\ge p\\
0&j\le p-1.
\end{cases}
\e
Moreover, observe that the coefficients $\al^{(p)}_n(A)$ satisfy \cite{FS65} for all $p\ge1$
\be\label{eq2a:sec1}
\al^{(p)}_n(K)=\fract{1}{n}\sum^{n-1}_{j=0}(-1)^{j+1}\al^{(p)}_j(K)\nu^{(p)}_{n-j},\quad\al^{(p)}_0(K)=1.
\e
It is shown in \cite{BS05, ISN00} that the $p$-modified Fredholm determinants $\sideset{}{_p}\det(1+z K)$ are entire functions of $z$ whose order of zeros correspond in algebraic multiplicities to eigenvalues of the compact operator $K$. From Lidskii's theorem \cite{BS05, ISN00}, we have for any operator $K\in\mathcal{J}_1(\mathcal{H})$
\bee
\mbox{tr}(K)=\sum^\infty_{n=1}\la_n,
\ee
where $\la_n$ are eigenvalues of the operator $K$ with possibly accumulation point at zero, and
\bee
\mbox{tr}(K)=\sum^\infty_{n=1}(u_n,Ku_n)
\ee
for any orthonormal basis $\{u_n\}_{n\ge1}$ of $\mathcal{H}$.
This implies that for any $K\in\mathcal{J}_p(\mathcal{H})$, we have 
\bee
\mbox{tr}(K^j)=\sum^\infty_{n=1}\la^j_n
\ee
since $K^j\in\mathcal{J}_1(\mathcal{H})$ for all $j\ge p$ (cf.~\cite[Chap IV.~11]{ISN00}). There are several constructions for the $p$-modified Fredholm determinants but the one we use along with \eqref{eq1:sec1} is given by (cf.~\cite{BS05})
\be\label{eq4a:sec1}
\sideset{}{_p}\det(1+z K)=\prod^\infty_{n=1}\left[(1+z\la_n)\exp\left(\sum^{p-1}_{j=1}z^j(-\la_n)^j/j\right)\right].
\e
Suppose that $K\in\mathcal{J}_{p-1}(\mathcal{H})$ then it follows from \eqref{eq4a:sec1} that (cf.~\cite{ISN00,BS05})
\be\label{eq5:sec1}
 \sideset{}{_p}\det(1+z K)=\sideset{}{_{p-1}}\det(1+z K)\exp\left((-z)^{p-1}\mbox{tr}(K^{p-1})/(p-1)\right).
\e 
Thus if $K\in\mathcal{J}_1(\mathcal{H})$ then from \eqref{eq5:sec1} we have (cf.~\cite{BS05,BS77,ISN00})
\bee
\sideset{}{_2}\det(1+z K)=\sideset{}{_1}\det(1+z K)\exp\left(-z~\mbox{tr}(K)\right).
\ee
By a trace class operator, we refer to any operator $K\in\mathcal{J}_1(\mathcal{H})$, and by a Hilbert--Schmidt any operator $K\in\mathcal{J}_2(\mathcal{H})$. Let $X\subseteq\R$, and consider the integral operator $K$ defined on $\mathcal{H}$ by
\be\label{eq7:sec1}
Ku(x)=\int_X k(x,y)u(y)\mathrm{d}x,~x\in X,
\e
where $k(x,y)\in L^2(X\times X)$. Then $K$ is compact \cite{HH73, FS65}, since $\norm{K}_2=\norm{k}_{L^2(X\times X)}<\infty$. Unfortunately, such a characterisation for trace class operator is not available. Nevertheless, if the integral operator $K$ is of trace class and induced by a continuous kernel then its trace satisfies (cf.~\cite{BS05,BS77,ISN00})
\be\label{eq7a:sec1} 
\mbox{tr}(K)=\int_X k(x,x)\mathrm{d}x.
\e
For such an operator $K$, its determinant $\det_1(1+z K)$ is called the Fredholm determinant.
The Fredholm determinant, as introduced originally in~\cite{Fred}, is an entire function of $z$ which characterise the solvability of the following equation
\bee
u(x)+z Ku(x)=f(x),\quad x\in X,
\ee
where $f$ and the integral kernel $k$ associated with $K$ are both assumed to be continuous in $X$. The latter is given by \cite{FN60,BS05,HH73,FS65}
\bee
d(z)=\sum^\infty_{n=0}z^n\al_n(K),
\ee 
where 
\be\label{eq8:sec1}
\al_n(K)=\fract{1}{n!}\int_{X^n}\det(k(x_p,x_q))^n_{p,q=1}\mathrm{d}x_1\cdots \mathrm{d}x_n.
\e
Note that \eqref{eq8:sec1} makes perfectly sense for any continuous kernel independently of whether $K$ is of trace class or not.
\section{Convergence analysis}
Let $X=[a,b]$ and consider the following eigenvalue problem
\be\label{eq1aa:sec2}
(I_\mathcal{H}+zK)u(x)=0,\quad x\in[a,b],
\e
where $z\in\C$ and $K$, defined by equation \eqref{eq7:sec1}, is an integral operator acting on $\mathcal{H}$. Basically to solve problem \eqref{eq1aa:sec2}, there are two preferred methods; the expansion or projection and the quadrature method (Nystr\"om-type method). But for the sake of simplicity, we rather choose the latter method since its implementation is straightforward and it is efficient for smooth kernels (cf.~\cite{FB10}). However, the results in this paper are also applicable for the projection method. This is under the hypotheses stated in \cite{GV67} which roughly assumes that $K$ is compact and that both $\norm{K-P_nK}$ and $\norm{K_n-P_nK}$ tend to zero as $n\to\infty$, where $P_n$ is a projection operator. The type of kernel function $k(x,y)$ that we consider in this paper are continuous everywhere in the domain except on the diagonal, i.e. the set $\{(x,y)\in[a,b]\times[a,b]\colon x=y\}$, and they satisfy the following hypotheses.
\begin{guess}\label{guess1}
 Write $k_x(y)=k_y(x)=k(x,y)$. Assume that for all $x\in[a,b]$, $k_x\in L^1[a,b]$,
\bee
\sup_{x\in[a,b]}\norm{k_x}_{L^1}<\infty 
\ee
and for $x_1,x_2\in[a,b]$
\be\label{eq1a:sec2}
\norm{k_{x_{1}}-k_{x_{2}}}_{L^1}\to0,~\mbox{as}~x_1\to x_2. 
\e
\end{guess}
Observe that, we  have for $x_1,x_2\in[a,b]$ and using \eqref{eq1a:sec2} that
\bee
 \abs{Ku(x_1)-Ku(x_2)}\le\norm{u}_\mathcal{H}~\norm{k_{x_1}-k_{x_2}}_{L^1}
\ee
Hence, it follows that the operator $K$ maps $\mathcal{H}$ into the space of continuous functions, $C[a,b]$.
Let $B=\{u\in\mathcal{H}:~\norm{u}_\mathcal{H}\le1\}$. Under Hypothesis \ref{guess1}, the integral operator $K$ is compact. Indeed, the set $K(B)\subset C[a,b]$ is bounded in $\mathcal{H}$ since for all $u\in B$
\bee
\norm{Ku}_\mathcal{H}<\infty, 
\ee
which follows from finiteness of the domain and the fact that
\bee
\abs{Ku(x)}\le\sup_{x\in[a,b]}\norm{k_x}_{L^1},~\mbox{for all}~x\in[a,b]~\mbox{and}~u\in B. 
\ee
Moreover, for $x_1,x_2\in[a,b]$
\bee
\abs{Ku(x_1)-Ku(x_2)}\le\norm{k_{x_1}-k_{x_2}}_{L^1}\to0,~\mbox{as}~x_1\to x_2.
\ee
The Arzel\`a--Ascoli lemma tells us that a totally bounded set in $C[a,b]$ is a bounded equicontinuous family of functions. Hence, the set $K(B)$ is totally bounded which implies that $K$ is compact. Henceforth, we assume that the kernel function $k(x,y)$ satisfying Hypothesis \ref{guess1} is given by
\be\label{eq2aa:sec2}
k(x,y)=g(x,y)h(x,y),
\e
where $g(x,y)$ satisfies Hypothesis \ref{guess1} and $h(x,y)$ is continuous everywhere in the domain. Note that with Hypothesis \ref{guess1}, we can consider $g(x,y)=\abs{x-y}^{-\al}$ with $0\le\al<1$.
Assume for fixed $x_n\in[a,b]$, $u\in C[a,b]$ and $N$ sufficiently large that
\be\label{eq2a:sec2}
h(x_n,y)u(y)\approx[h(x_n,y)u(y)]_N=\sum^N_{j=1}a_{nj}P_j(y),
\e
where the $P_j$ are a polynomial basis for $C[a,b]$, preferably orthogonal. Then we define the operator $K_N$ by
\be\label{eq3:sec2}
K_Nu(x)=\sum^N_{j=1}w_j(x)h(x,y_j)u(y_j),
\e
where for fixed $x_n\in[a,b]$
\bee
w_l(x_n)=\sum^{N}_{j=1}\beta_j(x_n)\left(P_j(y_l)\right)^{-1},~l=1,\dots,N
\ee
with
\bee
\beta_j(x_n)=\int^b_ag(x_n,y)P_j(y)\mathrm{d}y.
\ee
The operator $K_N$ maps $\mathcal{H}$ into $C[a,b]$, as well. This is because for all $x\in[a,b]$, the continuity of $w_j$ follows from that of $\beta_j$ since by assumption the function $g$ satisfies equation \eqref{eq1a:sec2}. Moreover, by assumption $h$ is continuous hence the product $w_j(x)h(x,y_j)$ is continuous. Thus for all $u\in B$ and $x\in [a,b]$, we have
\begin{align}
 \abs{K_Nu(x)}&\le\sum^N_{j=1}\abs{w_j(x)h(x,y_j)}\nonumber\\
&\le\max_{x\in[a,b]}\sum^N_{j=1}\abs{w_j(x)h(x,y_j)}\nonumber\\
&\le \sup_{x\in[a,b]}\norm{g_x}_{L^1}\max_{x\in[a,b]}\abs{h(x,y_j)}<\infty,~\forall N\ge1.\label{eq3aa:sec2}
\end{align}
Since if $h(x,y)=1$ and $u(x)=1$ together with the assumption \eqref{eq2a:sec2}, we have as $N$ goes to $\infty$
\bee
\max_{x\in[a,b]}\sum^N_{j=1}\abs{w_j(x)}\to\sup_{x\in[a,b]}\norm{g_x}_{L^1}<\infty
\ee
\begin{rem}
A quadrature method as described in \cite{KA67b} can be applied in order to approximate $h(x_l,y)u(y)$ instead of the projection method \eqref{eq2a:sec2}.
\end{rem}
Given the operator $K_N$, equation~\eqref{eq1aa:sec2} is approximated by
\be\label{eq3a:sec2}
u_N(x)=zK_Nu_N(x)=z\sum^N_{j=1}w_j(x)h(x,y_j)u_N(y_j),\quad x\in [a,b].
\e
The Nystr\"om-type method (cf.~\cite{EJN30}) is then obtained by substituting $x=y_i$ in equation \eqref{eq3a:sec2}. This yields a finite dimensional eigenvalue problem given by
\be\label{eq4:sec2}
u_N(y_i)=zK_Nu_N(y_j)=z\sum^N_{j=1}w_j(y_i)h(y_i,y_j)u_N(y_j),~i=1,\dots,N.
\e
To illustrate the use of the $p$-modified Fredholm determinants, suppose that $K$ is of trace class and $g(x,y)\equiv1$ for all $x,y\in[a,b]$, in \eqref{eq2aa:sec2}. In that case, $w_j(x)=w_j$ is constant for all $x\in[a,b]$ in \eqref{eq3a:sec2}. From equation \eqref{eq4:sec2}, one can then deduce that eigenvalues of the operator $K_N$ are precisely the zeros of the function $d_N(z)$ defined by
\be\label{eq4a:sec2}
d_N(z)=\sideset{}{_{\C^N}}\det(I_N-zK_N).
\e
This follows from the fact that
\bee
 \sideset{}{_{\C^n}}\det(A-\la I_n)=(-\la)^n\sideset{}{_{\C^n}}\det(I_n-zA),
\ee
for a linear operator $A\in \C^{n\times n}$ and $\la\ne 0$, where $z=\la^{-1}$. From arguments given in~\cite[Chap VI]{FS65}, we can write
\be\label{eq5:sec2}
\sideset{}{_{\C^N}}\det(I_N-zK_N)=\sum^\infty_{n=0}(-1)^n\al_{nN}z^n,
\e
where $\al_{nN}$ is defined as in \eqref{eq2:sec1} and \eqref{eq3:sec1} with \m{K^n} replaced by \mm{\C^N}{K^n_N}. In fact, as mentioned in \cite{FB10} and proved in \cite{CDM00}, the series \eqref{eq5:sec2} must terminate at $n=N$. Thus, under the assumption that $K$ is of trace class associated with a continuous kernel $h$, we may write 
\be\label{eq5a:sec2}
d_N(z)=\prod^N_{n=1}\left(1-z\la_{nN}(K_N)\right),
\e
where $\la_{nN}(K_N)$ are eigenvalues of the operator $K_N$, for $n=1,\cdots,N$. Since by assumption the kernel $h$ is continuous everywhere in the domain then \mm{\C^N}{K_N} converges to a finite \m{K} which is defined by equation \eqref{eq7a:sec1}. Hence for bounded $z\in\C$, the determinant given in \eqref{eq5a:sec2} or \eqref{eq5:sec2} converges uniformly to the Fredholm determinant $d(z)$~(cf.~\cite[Theorem~6.1]{FB10}). It then follows that eigenvalues of $K_N$ converge to that of $K$. To generalise Theorem~6.1 of~\cite{FB10} other than for continuous kernel, we assume that the operator $K_N$ for all $N\ge1$ satisfy Anselone's hypotheses of collectively compact operator \cite{ANS71}. That is, a set $\{K_n\}_{n\ge1}$ is called a collectively compact if
\begin{itemize}
 \item [A1.]$K$ and $K_n$ are linear operators on the Banach space $\mathcal{B}$ into itself.
\item [A2.]$K_nu\to Ku$ as $n\to\infty$, for all $u\in \mathcal{B}$ and $n\ge1$.
\item [A3.]The set $\{K_n\}_{n\ge1}$ is collectively compact that is, $\{K_nu:~n\ge1,~\norm{u}\le1\}$ has compact closure in $\mathcal{B}$.
\end{itemize}
Let $\mathcal{L}(\mathcal{H})$ denote the set of bounded linear operators on $\mathcal{H}$.
\begin{Def}[Anselone~\cite{ANS71}]\label{def1}
A set $\mathcal{K}\subset\mathcal{L}(\mathcal{H})$ is collectively compact provided that the set
\bee 
\mathcal{K}B=\{Ku:~K\in\mathcal{H},~u\in B\}
\ee
is relatively compact\footnote[1]{Anselone \cite{ANS71}: Relatively compact, sequentially compact and totally bounded are equivalent in a complete space}. A sequence of operators in $\mathcal{L}(\mathcal{H})$ is collectively compact whenever the corresponding set is. 
\end{Def}
Under Hypostesis \ref{guess1}, the set $\{K_N\}_{N\ge1}$ is collectively compact. Indeed, the operator $K_N$ defined in \eqref{eq3:sec2} satisfies A1. Regarding A2, observe that 
\bee
\abs{Ku(x)-K_Nu(x)}\le M\max_{x\in[a,b]}\int^b_a \abs{h(x,y)u(y)-[h(x,y)u(y)]_N}\mathrm{d}y,
\ee
where $M=\sup_{x\in[a,b]}\norm{g_x}_{L^1}$. Given the assumption \eqref{eq2a:sec2}, we have for sufficiently large $N$ 
\bee 
K_Nu\to Ku.
\ee
For $u\in B$, we have by equation \eqref{eq3aa:sec2} that 
\bee
\norm{K_Nu}\le M\max_{x\in[a,b]}\abs{h(x,y_j)}<\infty,~\forall N\ge1.
\ee
This shows the uniform boundedness of $\{K_N\}$ with 
\bee
\norm{K_N}\le M\max_{x\in[a,b]}\abs{h(x,y_j)}<\infty,~\forall N\ge1.
\ee 
Furthermore, for $x_1,x_2\in[a,b]$, we have
\begin{align*}
\abs{K_Nu(x_1)-K_Nu(x_2)}&\le\sum^N_{j=1}\abs{w_j(x_1)h(x_1,y_j)-w_j(x_2)h(x_2,y_j)}\\
&\le\sum^N_{j=1}\Bigl(\abs{(w_j(x_1)-w_j(x_2))h(x_1,y_j)}\Bigr.\\
&\Bigl.\qquad~+\abs{(h(x_1,y_j)-h(x_2,y_j))w_j(x_2)}\Bigr)\\
&\le\sum^N_{j=1}\Bigl(\norm{h}~\abs{w_j(x_1)-w_j(x_2)}\Bigr.\\
&\Bigl.\qquad~+\norm{w_j}~\abs{h(x_1,y_j)-h(x_2,y_j)}\Bigr)
\end{align*}
Since the functions $w_j$ and $h$ are continuous, it then follows that for all $N\ge1$ 
\bee
\abs{K_Nu(x_1)-K_Nu(x_2)}\to0~\mbox{as}~x_1\to x_2.
\ee
Hence, by the Arzel\`a--Ascoli lemma again, the set $\{K_N(B)\}_{N\ge1}$ is relatively compact. Therefore from Definition \ref{def1}, the set $\{K_N\}_{N\ge1}$ is collectively compact. Having proved the compactness property of the operators $K$ and $\{K_N\}_{N\ge1}$, we are now ready to generalise Theorem 6.1 of \cite{FB10} to integral operators belonging to $\mathcal{J}_p(\mathcal{H})$ for $p\ge1$. Under the assumption that $\{K_N\}_{N\ge1}$, obtained whether by quadrature or by projection method, is collectively compact then for $N$ sufficiently large each eigenvalue $\la$ of $K$ is a limit of a sequence of eigenvalues $\la_N$ of $K_N$ (cf.~\cite{KA67,GV67}). It then follows that the $k$th power of $\la$ is also a limit of the $k$th power of the sequence $\la_N$. Hence for bounded $z$, one gets a uniform convergence since the error in approximating the $p$-modified Fredholm determinants depend on the error in approximating each eigenvalue $\la$ of $K$ which is uniform depending on the approximation error $Ku-K_Nu$. However, there exists some special type of integral operators $K$ that do not fully satisfy Hypothesis~\ref{guess1} but are in $\mathcal{J}_4(\mathcal{H})$ (cf.~\cite{HB69}). Consider for instance, operators with kernel $k(x,y)=H(x,y)\abs{x-y}^{-\al}$ where $H(x,y)$ is assumed to be continuous and $\al=3/2$ (cf.~\cite{HB69}). In that case, the convergence analysis of this paper no longer applies since $w_j(x)$ is not Lebesgue integrable. However one strategy that we could use, as long as we are only interested in eigenvalues, is to compute the Fredholm determinant of the $p$th power of the operator $K$. This because for some $p$, the $p$th power integral operaor $K^p$ is of trace class and its kernel is continuous. Consequently, it may happen that in some case a relationship between the Fredholm determinant corresponding to $K^p$ and the $p$-modified Fredholm determinants associated with the operator $K$ could be established. This is of course will be possible depending on the the properties of $\sideset{}{_p}\det(I_\mathcal{H}+zK)$. In particular, if $K\in\mathcal{J}_2(\mathcal{H})$ or $K\in\mathcal{J}_3(\mathcal{H})(K\in\mathcal{J}_4(\mathcal{H}))$ we have the following theorem.
\begin{thm}\label{theom1a}
Suppose that $K\in\mathcal{J}_2(\mathcal{H})$ then 
\bee
\sideset{}{_1}\det(I_\mathcal{H}-z^2K^2)=\sideset{}{_2}\det(I_\mathcal{H}-zK)\sideset{}{_2}\det(I_\mathcal{H}+zK).
\ee
Moreover, if $\sideset{}{_2}\det(I_\mathcal{H}-zK)=\sideset{}{_2}\det(I_\mathcal{H}+zK)$ then 
\bee
(\sideset{}{_2}\det(I_\mathcal{H}-zK))^2=\sideset{}{_1}\det(I_\mathcal{H}-z^2K^2).
\ee
If $K\in\mathcal{J}_3(\mathcal{H})$ then $K^2\in\mathcal{J}_2(\mathcal{H})$ and we have
\bee
\sideset{}{_2}\det(I_\mathcal{H}-z^2K^2)=\sideset{}{_3}\det(I_\mathcal{H}-zK)\sideset{}{_3}\det(I_\mathcal{H}+zK)
\ee
\end{thm} 
\begin{proof}
Since the product of two Hilbert--Schmidt operators is of trace class, we have
\begin{align*}
\sideset{}{_1}\det(I_\mathcal{H}-z^2K^2)&=\prod^\infty_{n=1}(1-z^2\la^2_n)\\
&=\prod^\infty_{n=1}\Bigl((1-z\la_n)\exp(\la_nz)(1+z\la_n)\exp(-\la_nz)\Bigr)\\
&=\prod^\infty_{n=1}\Bigl((1-z\la_n)\exp(\la_nz)\Bigr)\prod^\infty_{n=1}\Bigl((1+z\la_n)\exp(-\la_nz)\Bigr)\\
&=\sideset{}{_2}\det(I_\mathcal{H}-zK)\sideset{}{_2}\det(I_\mathcal{H}+zK)\\
&=(\sideset{}{_2}\det(I_\mathcal{H}-zK))^2
\end{align*}
Now if $K\in\mathcal{J}_3(\mathcal{H})$ then observe that $K^2$ satisfies (cf. \cite[Theorem 11.2, Chap IV]{ISN00})
\be\label{eq6:sec2}
\text{tr}(\abs{K^2}^2)\le\text{tr}(\abs{K}^4)\le(\text{tr}(\abs{K}^3))^{4/3}<\infty.
\e
Hence $K^2$ is Hilbert--Schmidt. It then follows that
\begin{align*}
\sideset{}{_2}\det(I_\mathcal{H}-z^2K^2)&=\prod^\infty_{n=1}(1-z^2\la^2_n)\exp(z^2\la^2_n)\\
&=\prod^\infty_{n=1}\Bigl((1-z\la_n)\exp(\la_nz+z^2\la^2_n/2)\Bigr.\\
&\qquad\quad\times\Bigl.(1+z\la_n)\exp(-\la_nz+z^2\la^2_n/2)\Bigr)\\
&=\sideset{}{_3}\det(I_\mathcal{H}-zK)\sideset{}{_3}\det(I_\mathcal{H}+zK)
\end{align*}
\end{proof}
\begin{rem}
 If $K\in\mathcal{J}_4(\mathcal{H})$ then from \eqref{eq6:sec2}, we have
 \bee
\sideset{}{_2}\det(I_\mathcal{H}-z^2K^2)=\sideset{}{_4}\det(I_\mathcal{H}-zK)\sideset{}{_4}\det(I_\mathcal{H}+zK).
\ee
And if $\sideset{}{_4}\det(I_\mathcal{H}-zK)=\sideset{}{_4}\det(I_\mathcal{H}+zK)$ then we have 
\bee 
\sideset{}{_2}\det(I_\mathcal{H}-z^2K^2)=(\sideset{}{_4}\det(I_\mathcal{H}-zK))^2.
\ee
More generally for $K\in\mathcal{J}_{2p}(\mathcal{H})$, we have for $\text{Re}(z)>0$ and $p=3,4,\dots$ 
\bee
\sideset{}{_{2p}}\det(I_\mathcal{H}-zK)\sideset{}{_{2p}}\det(I_\mathcal{H}+K)=c_{p}(z)\sideset{}{_{\lceil p/2\rceil}}\det(I_\mathcal{H}-z^4K^4),
\ee
where $c_{p}(z)=1/\sideset{}{_{p}}\det(I_\mathcal{H}+z^2K^2)$. The above equality holds also for $K\in\mathcal{J}_{2p-1}(\mathcal{H})$.
\end{rem}
In what follows, we define for all $p\ge1$, $d_p(z)\coloneqq\sideset{}{_p}\det(1+zK)$ and
\bee
d_{Np}(z)\coloneqq\sideset{}{_{p\C^N}}\det(I_N+zK_N)=\sum^\infty_{k=0}\al^{(p)}_{kN}(K_N)z^k.
\ee
\begin{thm}\label{theom1}
 Suppose that $K$ given in \eqref{eq7:sec1} is in $\mathcal{J}_p(\mathcal{H})$ and $\{K_N\}_{N\ge1}$ is collectively compact with $K_N$ defined by \eqref{eq3:sec2}. Then for $p\ge 1$, we have
\bee
d_{Np}(z)\to d_p(z) 
\ee
converges uniformly for all bounded $z$ as $N\to \infty$.
\end{thm}
\begin{proof}
Let $z\in\C$ be bounded by $M>0$. Then we have
\bee 
\abs{d_p(z)-d_{Np}(z)}\le \sum^\infty_{j=0}\abs{\al^{(p)}_j(K)-\al^{(p)}_{jN}(K_N)}~M^j.
\ee
Observe that from \eqref{eq2a:sec1} and \eqref{eq3:sec1}, the error $\abs{\al^{(p)}_j(K)-\al^{(p)}_{jN}(K_N)}$ can be deduced for all $j\ge p$ from 
\begin{align*}
 \abs{\mbox{tr}(K^j)-\mbox{tr}_{\C^N}(K^j_N)}&=\Abs{\sum^\infty_{n=1}\la^j_n-\sum^N_{n=1}\la^j_{nN}}\\
&\le\Abs{\sum^N_{n=1}\la^j_n-\sum^N_{n=1}\la^j_{nN}}+\Abs{\sum^\infty_{n=N+1}\la_n^j},
\end{align*}
where $\la_n$ and $\la_{nN}$ are eigenvalues of the operators $K$ and $K_N$, respectively.\\
As $N$ goes to infinity, we have
\bee
\sum^\infty_{n=N+1}\abs{\la_n}^j\to0. 
\ee
Therefore
\bee
\Abs{\sum^\infty_{n=N+1}\la^j_n}\le \sum^\infty_{n=N+1}\abs{\la_n}^j\to0.
\ee
On the other hand, since $\{K_N\}_{N\ge1}$ is collectively compact, we have for some $N$ sufficiently large, $\ep>0$ and for each $n\in\{1,\dots,N\}$ (cf.~\cite{KA67,GV67} and \cite[Theorem 4.8]{ANS71})
\bee
\abs{\la_n-\la_{nN}}\le\ep.
\ee
Observe that 
\begin{align}
\abs{\la^j_n-\la^j_{nN}}&=\abs{\la_n-\la_{nN}}\Abs{\sum^{j-1}_{m=0}\la^m_n\la^{j-m}_{nN}}\nonumber\\
&\le \abs{\la_n-\la_{nN}}\sum^{j-1}_{m=0}\abs{\la^m_n\la^{j-m}_{nN}}\label{eq7a:sec2}\\
&\le j\abs{\la_n}^j\ep\label{eq7c:sec2}.
\end{align}
This is because for some $N$ sufficiently large, we have
\be\label{eq7b:sec2}
\sum^{j-1}_{m=0}\la^m_n\la^{j-m}_{nN}\approx j\abs{\la_n}^j.
\e
It then follows from \eqref{eq7c:sec2} and  the continuous embedding of $\mathcal{J}_p(\mathcal{H})\subset\mathcal{J}_q(\mathcal{H})$ for $p<q$ that
\bee
\sum^N_{n=1}\abs{\la^j_n-\la^j_{nN}}\le\ep j\sum^N_{n=1}\abs{\la_n}^j\le \ep j\norm{K}^p_p.
\ee
Thus for some $N$ large and $\ep$ chosen arbitrarily small, we have
\bee
\abs{\mbox{tr}(K^j)-\mbox{tr}_{\C^N}(K^j_N)}\le\ep.
\ee
Consequently, we have $\abs{d_p(z)-d_{Np}(z)}\le\ep$ as $N\to\infty$.
\end{proof}

\begin{rem}
 Suppose that $k(x,y)$ is continuous on $[a,b]^2$ then (cf.~\cite{HH73}) we have
\be\label{eq8:sec2}
\mbox{tr}(K^n)=\int^b_ak_n(x,x)\mathrm{d}x,
\e
where $k_1(x,y)=k(x,y)$ for all $x,y\in[a,b]$ and for $n\ge2$
\bee
k_n(x,x)=\int_{[a,b]^{n-1}}k(x,x_{n-1})k_{n-1}(x_{n-1},x)\mathrm{d}x_{n-1}.
\ee
Note that in the process of computing the finite dimensional determinant \eqref{eq5:sec2}, we essentially approximate the multiple integrals~\eqref{eq8:sec2} by a product quadrature rule $Q_n$, that is
\begin{align*}
Q_n(\mbox{tr}(K^n))&=\mbox{tr}_{\C^N}(K^n_N)\\
&= \sum^N_{j_1,\dotso,j_n=1}w_{j_1}\cdots w_{j_n}k(x_{j_n},x_{j_1})k(x_{j_1},x_{j_2})\cdots k(x_{j_{n-1}},x_{j_n}).
\end{align*}
Therefore the error becomes
\be\label{eq8a:sec2}
\abs{\mbox{tr}(K^n)-\mbox{tr}_{\C^N}(K^n_N)}=\abs{\mbox{tr}(K^n)-Q_n(\mbox{tr}(K^n))},
\e
which could be related to the quadrature error $\abs{Ku(x)-K_Nu(x)}$ for $x\in[a,b]$. Hence, if we assume that the kernel $k(x,y)$ is smooth then one might expect an exponential convergence of \eqref{eq8a:sec2} as mentioned in~\cite{FB10}.
\end{rem}
Under the assumption that $\{K_N\}_{N\ge1}$ is a collectively compact operator and exploiting the results of \cite{KA75, KA67, GV67}, we can now estimate the rate of convergence in evaluating the determinant $d_{Np}(z)$ at $z=\la^{-1}$, where $\la\in\sig(K)$, the spectrum of $K$. In fact, this estimate depends essentially on the error in approximating the integral operator $K$ in \eqref{eq7:sec1}. In what follows, $c$ is always a positive constant.
\begin{thm}\label{theom2}
 Assume that $\{K_N\}_{N\ge1}$ is collectively compact. Let $\la_{n_0}\ne0$ and $\la_{n_0N}$ denote the eigenvalues of $K\in\mathcal{J}_p(\mathcal{H})$ and $K_N$, respectively. Then for some $N$ sufficiently large and for $p\ge1$, we have
\be\label{eq9:sec2}
\abs{d_p(z_{n_0})-d_{Np}(z_{n_0})}\le c\max\{\norm{Ku_i-K_Nu_i}^{1/l}:~1\le i\le m\},
\e
where $z_{n_0}=\la^{-1}_{n_0},~\{u_1,\cdots,u_m\}$ is a basis for $\mathrm{Ker}(\la_{n_0}-K)^l$ and $m$ and $l$ are the multiplicity and the index of $\la_{n_0}$, respectively.
\end{thm}
\begin{proof}
For simplicity, we consider the case $p=1$. However the following proof holds for $p\ge2$, where we make use of the definition of the determinant given in \eqref{eq4a:sec1} for finite dimensional matrices. Evaluating the Fredholm determinant at its zero $z_{n_0}=\la^{-1}_{n_0}$, for some $1\le n_0\le N$ and $\la_{n_0}\ne0$, we have 
\begin{align}
 \abs{d(z_{n_0})-d_{N}(z_{n_0})}&=\abs{d_{N}(z_{n_0})}\nonumber\\
&=\Abs{\prod^N_{n=1}(1-\la^{-1}_{n_0}\la_{nN})}\nonumber\\
&=\Abs{\prod^N_{n=1}\left[\la^{-1}_{n_0}\left(\la_{n_0}-\la_{nN}\right)\right]}\nonumber\\
&=\Abs{\la_{n_0}}^{-N}\prod^N_{n=1}\abs{\la_{n_0}-\la_{nN}}\label{eq9a:sec2}.
\end{align}
Since $\{K_N\}_{N\ge1}$ is a collectively compact operator, it then follows from~\cite{KA75} and \cite[Theorem 3]{GV67} when $n=n_0$ in the right-hand  side of \eqref{eq9a:sec2} that
\bee
 \abs{d(z_{n_0})-d_{N}(z_{n_0})}\le c\max\{\norm{Ku_i-K_Nu_i}^{1/l}: 1\le i\le m\},
\ee
for some $N$ sufficiently large.
\end{proof}
For the projection method, it suffices to replace the right-hand side of \eqref{eq9:sec2} by the error bound in \cite[Theorem~3]{GV67}. 
\begin{rem}
 Under normalising the eigenfunction $u_i$, the set $\{u_1,\cdots,u_m\}\subset B$. On the other hand, since the set $\{K_N(B)\}_{N\ge1}$ and $K(B)$ are totally bounded and mapping $\mathcal{H}$ into $C[a,b]$, then, $K_Nu_i$ converge uniformly to $Ku_i$ \cite[Proposition~1.7]{ANS71}, that is as $N\to\infty$
 \bee
 \norm{Ku_i-K_Nu_i}\to0.
 \ee
\end{rem}
\begin{thm}
 Suppose $K$ and $\{K_N\}_{N\ge1}$ are as in Theorem \ref{theom2}. Then for some $N$ sufficiently large and $z$, an element of the resolvent set $\rho(K)$, we have
\be\label{eq9b:sec2}
\abs{d_p(z)-d_{Np}(z)}\le c\beta(N,l_{n_0})\norm{K}^p_p,
\e
where
\be\label{eq9c:sec2}
\beta(N,l_{n_0})=\max\{\norm{Ku_i-K_Nu_i}^{1/l_{n_0}}:1\le i\le m_{n_0}\}
\e
and $l_{n_0}=\max\{l_n\colon n=1,\dots,N\}$.
\end{thm}
\begin{proof}
Suppose that $z\in\rho(K)$. Then combining equation \eqref{eq7a:sec2} and \eqref{eq7b:sec2}, and the error estimate of \cite{KA75} for each $\la_n,~n=1,\cdots,N$, we have, for all $k\ge p$ and for some $N$ sufficiently large, 
\bee
\abs{\mbox{tr}(K^k)-\mbox{tr}_{\C^N}(K^k_N)}\le ck\sum^N_{n=1}\abs{\la_n}^{k}\max\{\norm{Ku_i-K_Nu_i}^{1/l_n}: 1\le i\le m_n\}.
\ee
Given $l_{n_0}$, we have for all $n=1,\dots,N$
\bee
\max\{\norm{Ku_i-K_Nu_i}^{1/l_n}\colon 1\le i\le m_n\}\le\beta(N,l_{n_0}), 
\ee
with $\beta(N,l_{n_0})$ given in \eqref{eq9c:sec2}. It then follows for sufficiently large $N$ that 
\bee 
\abs{\mbox{tr}(K^k)-\mbox{tr}_{\C^N}(K^k_N)}\le ck\beta(N,l_{n_0})\norm{K}^p_p.
\ee
Hence \eqref{eq9b:sec2} follows.
\end{proof}
Observe from equation \eqref{eq9b:sec2} and equation \eqref{eq9:sec2} that for a given eigenvalue $z^{-1}_0$ of index $l\le l_{n_0}$, we have for some $N$ sufficiently large and $z\in\rho(K)$, that
\be\label{eq10b:sec2}
\abs{d_p(z_0)-d_{Np}(z_0)}\le\abs{d_p(z)-d_{Np}(z)}. 
\e
Thus, we note from \eqref{eq9:sec2} and \eqref{eq9b:sec2} that the rate of convergence is controlled by the index of the eigenvalues. 
\begin{rem}
 Given the above inequality \eqref{eq10b:sec2}, we can observe that numerical computation of the $p$-modified Fredholm determinants is similar to the numerical interpolation in which the interpolation points are the eigenvalues $z^{-1}_n$ of $K$.
\end{rem}

\begin{cor}\label{cor1}
Let the integral operator $K\in\mathcal{J}_p(\mathcal{H})$ and the set $\{K_N\}_{N\ge1}$ is a collectively compact operator. Assume that eigenvalues $\la_n$ of the operator $K$ are simple, for all $n\ge1$, and that the associated eigenvector $u_{1}$ satisfies $\mathrm{Im}(u_{1})=0~(\mathrm{Re}(u_{1})=0)$. Then for some $N$ sufficiently large and for all $z\in\C~\mbox{and}~c>0$, we have
\be\label{eq10a:sec2}
\abs{d_p(z)-d_{Np}(z)}\le c\norm{Ku_{1}-K_Nu_{1}}.
\e
\end{cor}
Indeed, since the eigenvalues $\la_n$ of $K$ are simple then their algebraic multiplicities $m_n$ as well as their indexes $l_n$ are equal to $1$, for all $n\ge1$ \cite{CDM00}. It then follows that 
\bee
\beta(N,1)=\norm{Ku_1-K_Nu_1},
\ee 
and also from \eqref{eq10b:sec2} that
\bee
\abs{d_p(z)-d_{Np}(z)}=c\abs{d_p(z_n)-d_{Np}(z_n)} 
\ee
where $z^{-1}_n\in\sig(K)$. Hence, equation \eqref{eq10a:sec2} follows. In other words,  Corollary~\ref{cor1} tells us that the rate of convergence in computing the determinants is proportionally the same for any $z\in\C$, that is for $z^{-1}\in\sig(K)$ or $z\in\rho(K)$ (cf.~Figure~\ref{fig:1}). 

\begin{rem}\label{re3.5}
Suppose for example that $K$ in Corollary~\ref{cor1} is associated with a kernel which has a jump discontinuity on the diagonal in the first derivative. Then for a given quadrature method which does ignore the discontinuity region, the rate of convergence is the same for all $z\in\C$ like in Corollary~\ref{cor1}. However if the method takes into account the discontinuity then we might expect better convergence for $z^{-1}_0\in\sig(K)$ than for $z\in\rho(K)$ (cf.~Figure \ref{fig:1}).
\end{rem}
\section{Numerical Results}
In this section we numerically evaluate the $p$-modified Fredholm determinants associated with the operator $K$ given by \eqref{eq7:sec1}, where its corresponding kernel 
\be\label{eq1:sec4}
k(x,y)=
\begin{cases}
k_1(x,y),& a\le y\le x\\
k_2(x,y),&x<y\le b
\end{cases}
\e
is whether discontinuous along the diagonal or just continuous, and satisfies Hypothesis \ref{guess1}. For the numerical approach, we essentially follow the method described in \cite{KKR02}. To this end, let us now give a brief descprition of the method but more details are found in \cite{KKR02}. We assume that $k_1$ and $k_2$ can be approximated by Chebyshev polynomials $T_n(x)$, that is
\begin{subequations}
\begin{align*}
k_1(x_m,y)=&\sum^{N}_{n=1}a_{mn}T_n(y)\\
k_2(x_m,y)=&\sum^{N}_{n=1}\tilde{a}_{mn}T_n(y),
\end{align*}
\end{subequations}
for a fixed $x_m\in [-1,1]$ and $m=1,\cdots,N$. Then an approximation of the eigenvalue problem \eqref{eq1aa:sec2} is given by (cf.~\cite{KKR02})
\be\label{eq3:sec4}
\left(I_N+z\frac{b-a}{2}\left(CS_rC^{-1}\circ K_1+CS_lC^{-1}\circ K_2\right)\right)\bar{u}=0,
\e
where $\circ$ denotes the pointwise multiplication, $C=T_n(x_m),~K_1=k_1(x_m,x_n),$\\$K_2=k_2(x_m,x_n),~S_r$ and $S_l$ are right and left spectral integration matrix, respectively, and $\bar{u}=(u(x_1),\dots,u(x_N))^T$.  Hence
\bee 
d_{Np}(z)=\sideset{}{_N}\det(I_N+zK_N),
\ee
where 
\bee
K_Nu(x)=\sum^N_{j=1}\left(w_jk_1(x,y_j)+\al_jk_2(x,y_j)\right)u(y_j)
\ee
with $\{K_N\}_{N\ge1}$, a collectively compact operator. One of the advantages of using Chebyshev polynomials lie on the fact that the coefficients in the expansion of an indefinite integral can be easily obtain from that of the series expansion of the integrand, in terms of the Chebyshev polynomial~\cite{CC60}. Therefore with this property, the Chebyshev polynomials are extremely useful for integral equations associated with discontinuous kernel function along the diagonal. In all the examples below, we use the Nystr\"om--Clenshaw--Curtis and  the Nystr\"om--Gauss--Legendre referred to NCC as in \cite{KKR02} and NGL, respectively. The latter method does not take into account the discontinuity of the kernel function $k$ however the NCC does.
\begin{exmp}\label{exp:1}
For our first example, we consider the problem studied in \cite{FB10}, i.e
\bee 
u(x)=zKu(x)=\int^1_0k(x,y)u(y)\mathrm{d}y,~x\in [0,1],~u\in\mathcal{H}=L^2(0,1),
\ee
where
\bee 
k(x,y)=
\begin{cases}
 x(1-y),&x\le y\\
y(1-x),&y<x.
\end{cases}
\ee
Our aim for this problem is not to compute the Fredholm determinant but to emphasize Corollary~\ref{cor1} and Remark \ref{re3.5}. The operator $K$, for this example, is compact, self-adjoint and the corresponding set of operators $\{K_N\}_{n\ge1}$ is collectively compact since $k$ is continuous (cf.~\cite{ANS71}). Observe that eigenvalues of $K$, given by
\bee
z^{-1}_n=\la_n=\frac{1}{\pi^2n^2},~n=1,2,\dots,
\ee
are simple. Accordingly, Corollary~\ref{cor1} tells us that the rate of convergence is proportionally the same for all nonzero $z$, and it is of order $O(N^{-2})$ (cf.~\cite{FB10}). To confirm this, we display in Fig~\ref{fig:1} the error $\abs{d(z)-d_N(z)}$ computed by the NGL and the NCC method at the eigenvalue, $z_1=\la^{-1}_1=\pi^2$ and at $z=1$. Recall that at the roots $z_n$, the $p$-modified Fredholm determinants $d_p(z_n)=0$ so the error is just $\abs{d_{Np}(z_n)}$.
\begin{figure}[!h]
\centering
\includegraphics[height=6cm, width=6cm]{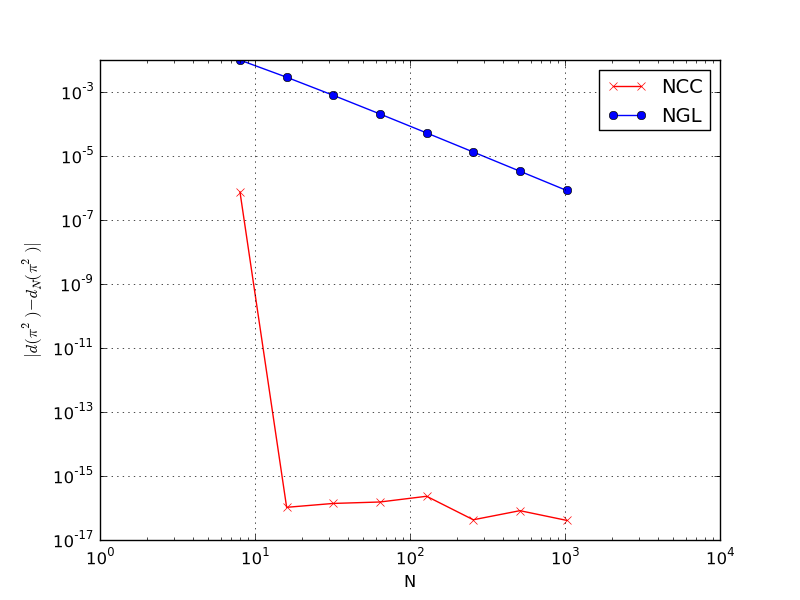}
\includegraphics[height=6cm, width=6cm]{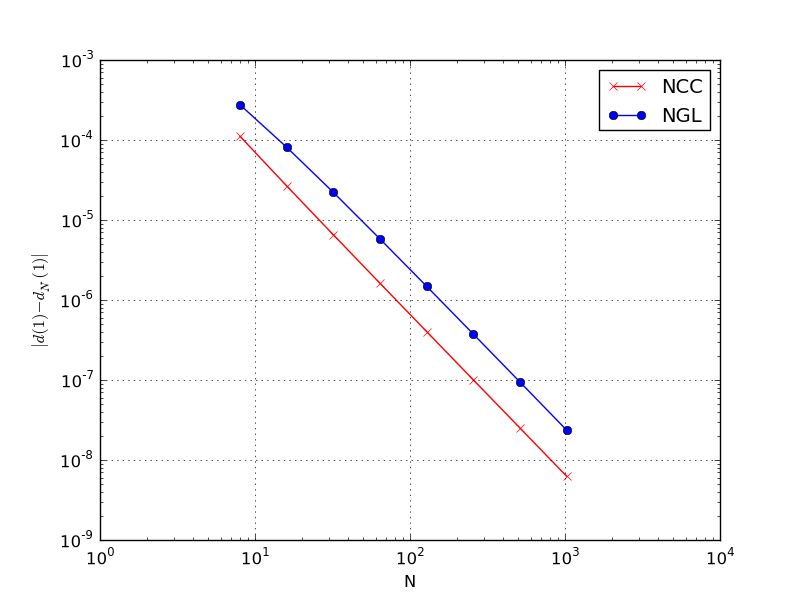}
\caption{$\log_{10}$ of the error evaluated at $z_1=\la^{-1}_1=\pi^2$ (left) and at $z=1$ (right) of Example~\ref{exp:1}.}
\label{fig:1}
\end{figure}
\end{exmp}

\begin{exmp}\label{exp:2}
Our aim for the second example is to show the different rates of convergence shown in \eqref{eq10b:sec2}, for $z\in\rho(K)$ and $z^{-1}\in\sig(K)$. The kernel function associated with the integral operator $K$ is given by
\bee
k(x,y)=\frac{1}{12}-\frac{1}{2}\abs{x-y}+\frac{1}{2}(x-y)^2,\quad x,y\in[0,1].
\ee
The operator $K$ is self-adjoint with eigenvalues given by (cf.~\cite{HH73})
\bee
z^{-1}_n=\la_n=\frac{1}{4n^2\pi^2},~n=1,2,\dotsc.
\ee
Since $\sum^\infty_{n=1}\abs{\la_n}<\infty$ then the operator $K$ is of trace class and its corresponding Fredholm determinant is
\bee
d(z)=\prod^\infty_{n=1}\left(1-\frac{z}{4n^2\pi^2}\right)=\fract{2-2\cos(\sqrt{z})}{z}.
\ee
Each eigenvalue $\la_n$ has an algebraic multiplicity $m_n=2$, and its corresponding index $l_n=1$ for all $n\ge1$, since $K$ is self-adjoint. Therefore following the same line of arguments in \cite{FB10} which led to an error of order $O(N^{-2})$ in Example \ref{exp:1}, we can conclude that 
\begin{align*}
\beta(N,1)&=\max\{\norm{Ku_{1}-K_Nu_{1}},\norm{Ku_{2}-K_Nu_{2}}\}\\
&=O(N^{-2})
\end{align*}
where $u_{1}(x)=\sqrt{2}\cos(2n\pi x)$ and $u_{2}(x)=\sqrt{2}\sin(2n\pi x)$ (cf.~\cite{HH73}). However,
\begin{figure}[!h]
\centering
\includegraphics[height=6cm, width=6cm]{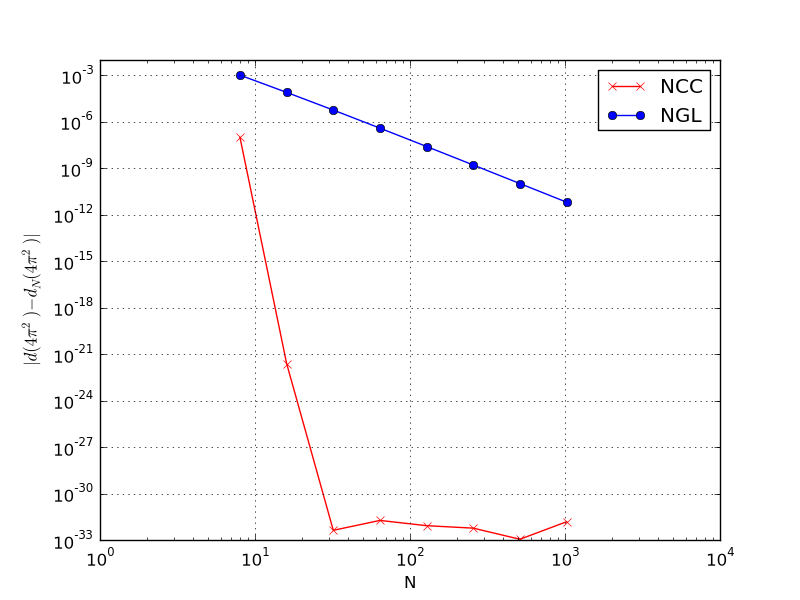}
\includegraphics[height=6cm, width=6cm]{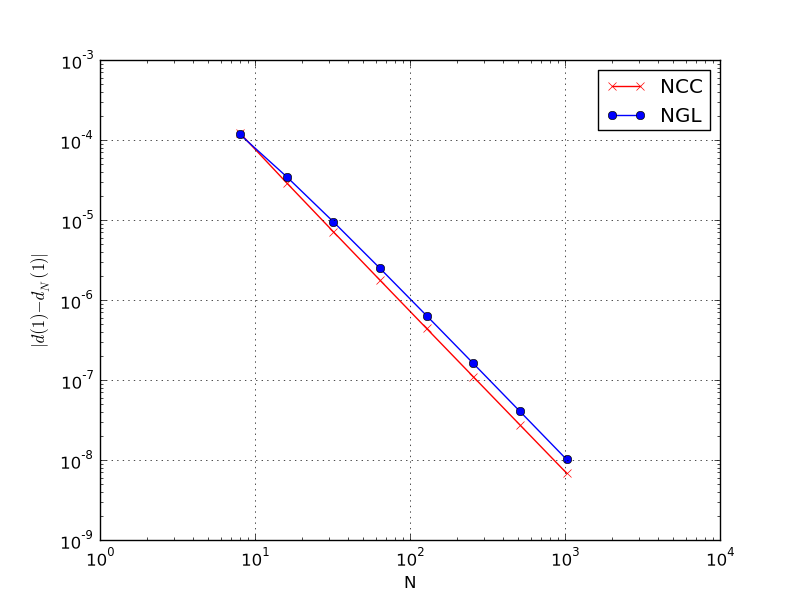}
\caption{$\log_{10}$ of the error evaluated at $z_1=\la^{-1}_1=4\pi^2$ (left) and at $z=1$ (right) for Example~\ref{exp:2}.}
\label{fig:2}
\end{figure}
numerical observation of the NGL method shows that the rate of convergence at the eigenvalue $z_{1}=4\pi^2$ is $O(N^{-4})$ (see Figure~\ref{fig:2} (left)), whereas at $z=1$ it is proportional to $O(N^{-2})$ which is also the rate for the NCC method (see Figure~\ref{fig:2} (right)).
\end{exmp}
\begin{exmp}\label{exp:3}
Let consider
\be\label{eq3a:sec4}
u(x)-z\int^1_{-1}k(x,y)u(y)\mathrm{d}y=0,~x\in [-1,1],~u\in\mathcal{H}=L^2(-1,1),
\e
where
\be\label{eq4:sec4}
k(x,y)=
\begin{cases}
 1,&y\le x\\
-1,&x<y.
\end{cases}
\e
Note that $k(x,y)$ is of the form given in \eqref{eq2aa:sec2} where $g(x,y)=\mbox{sgn}(x-y)$ and $h(x,y)=1$. The integral operator $K$ associated with the kernel function \eqref{eq4:sec4} is Hilbert--Schmidt operator since $\norm{k}_{L^2((-1,1)^2)}<\infty$, and it is a normal operator as well, that is $K^*K=KK^*$. The coefficients $\al^{(2)}_n(K)$ in the expression of the Hilbert--Schmidt determinant $d_2(z)$ are easily computed via \eqref{eq8:sec1} by setting to zero the diagonal elements (cf.\cite{DH04}), and are given explicitly, for $k=0,1,2,\dots,$ by
\bee
\al^{(2)}_n(K)=\frac{1}{n!}
\begin{cases}
 2^n,&n=2k\\
0,&n=2k+1.
\end{cases}
\ee
Hence the Hilbert--Schmidt determinant is 
\bee
d_2(z)=\sideset{}{_2}\det(I_\mathcal{H}-zK)=\cosh(2z)
\ee
with purely imaginary zeros $z_k$ satisfying
\be\label{eq5:sec4}
z^{-1}_n=\la_n=-i\frac{4}{\pi}\fract{1}{2n+1},\quad n=0,1,2,\dots~.
\e
The eigenfunctions $u_n$ associated with the eigenvalues $\la_n$ are
\be\label{eigfn}
u_n(x)=\frac{1}{\sqrt{2}}\exp\left(i\frac{\pi}{2}(2n+1)x\right). 
\e
From \eqref{eq5:sec4}, it is clear that the trace of $K$ is divergent but 
\be\label{eq6:sec4}
\mbox{tr}(K^2)=-2\fract{16}{\pi^2}\sum^\infty_{n=0}\fract{1}{(2n+1)^2}=-4. 
\e
The factor of $2$ before the sum in \eqref{eq6:sec4} takes into account the fact that each eigenvalue of the operator $K$ must be counted twice since $d_2(z_n)=d_2(\bar{z}_n)=0$. Observe that \eqref{eq6:sec4} can also be written as (cf.~\cite{FS65,HH73})
\bee
\mbox{tr}(K^2)=-4=\int^1_{-1}\int^1_{-1}k(x,y)k(y,x)\mathrm{d}x\mathrm{d}y.
\ee
From \eqref{eq4a:sec1}, we have 
\be\label{eq4a:sec4}
d_2(z)=\prod^\infty_{n=0}\left[(1-i\fract{4z}{\pi(2n+1)})\exp(i\fract{4z}{\pi(2n+1)})\right]=\cosh(2z).
\e
As a numerical proof of Theorem \ref{theom1}, we plot in Figure \ref{fig:3} the approximate determinant $d_{2N}(z)$ and the original $d_2(z)$ for $\abs{z}\le1$. In Figure~\ref{fig:3} (left), we plot the imaginary part of the Hilbert--Schmidt determinant and in Figure~\ref{fig:3} (right) the real part. One obviously note that as $N$ gets bigger, $d_{2N}(z)$ converge to $d_2(z)$ at the rate proportional to $O(N^{-1})$ due to the discontinuity in $k(x,y)$.
\begin{figure}[!h]
\centering
\includegraphics[height=6cm, width=6cm]{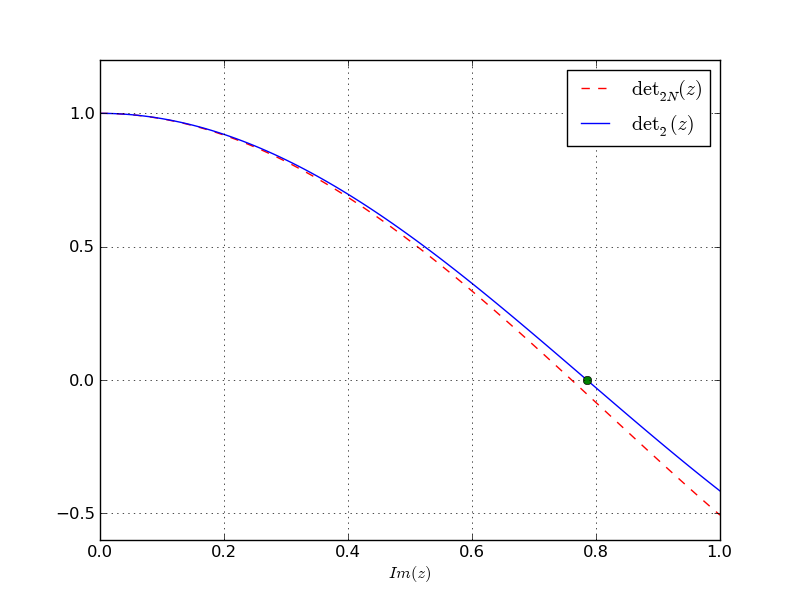}
\includegraphics[height=6cm, width=6cm]{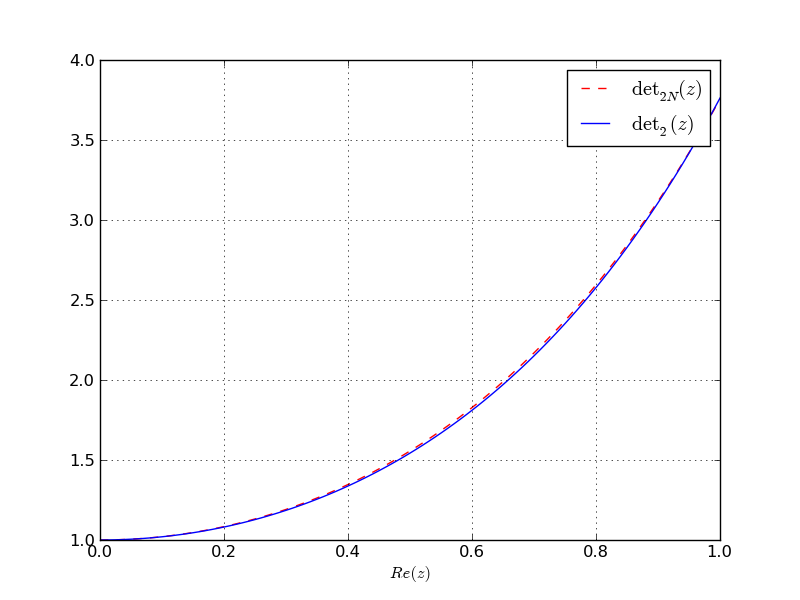}
\caption{Numerical approximation of $d_2(z)$ by rectangular rule for $N=30$. Left panel is the imaginary part and right panel is the real part.}
\label{fig:3}
\end{figure}
Again to confirm that \eqref{eq10b:sec2} hold, we plot the error $\abs{d_2(z)-d_{2N}(z)}$ at $z=1\in\rho(K)$ and $z_0=\la^{-1}_0=i\pi/4$. Since $K$ is a normal operator then the index of the eigenvalues $l_n=1$ for all $n\ge1$. Although, the eigenvalues of $K$ are simple but as the eigenfunctions \eqref{eigfn} are complex then Corollary \ref{cor1} does not hold for this example.
Numerically, we observe that the NGL method (dotted line) produces a rate of convergence proportional to $O(N^{-2})$ Figure~\ref{fig:4} (left) and $O(N^{-1})$ Figure~\ref{fig:4} (right), for $z_0=i\pi/4$ and $z=1$, respectively.
\begin{figure}[!h]
\centering
\includegraphics[height=6cm, width=6cm]{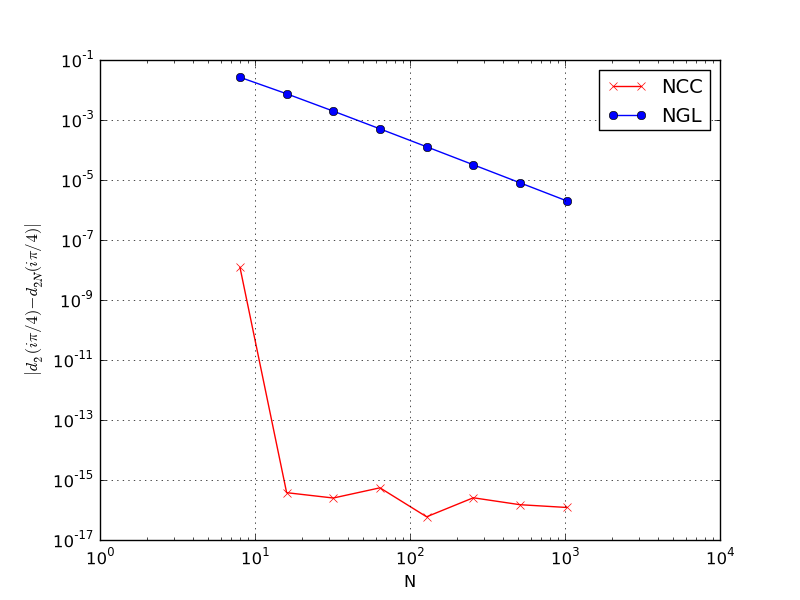}
\includegraphics[height=6cm, width=6cm]{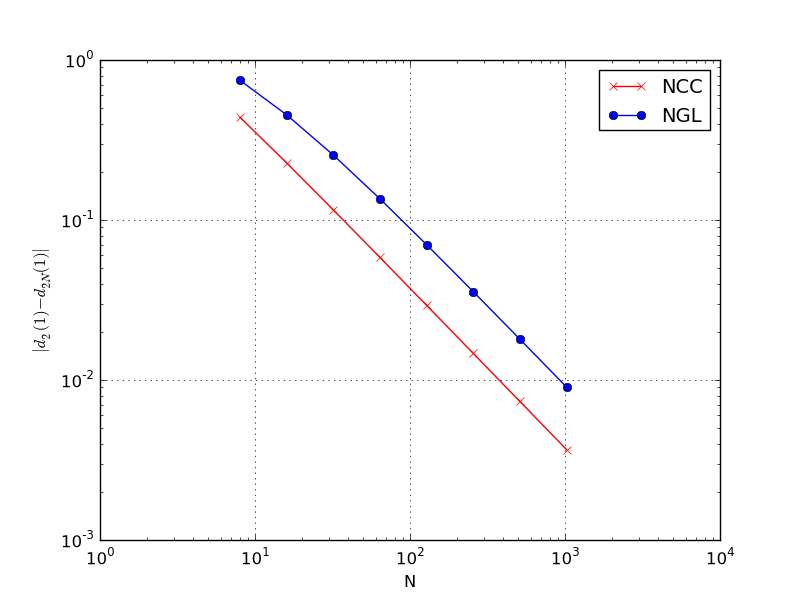}
\caption{$\log_{10}$ of the error evaluating at $z_1=\la^{-1}_1=i\pi/4$ (left) and at $z=1$ (right) for Example~\ref{exp:3}.}
\label{fig:4}
\end{figure}\\
To implement the strategy mentioned in the paragraph preceding Theorem \ref{theom1a}, we compute the iterated integral operator. Observe that the $2$-iterated kernel $k_2(x,y)$ of \eqref{eq4:sec4} is continuous Hermitian everywhere for $(x,y)\in [-1,1]\times[-1,1]$ and it is given by
\bee
k_2(x,y)=-2+2\abs{x-y}.
\ee
The integral operator $K^2$ associated with kernel $k_2$ is of trace class since $k_2$ satisfies condition of \cite[Theorem 8.2, Chap IV]{ISN00} with $\al>1/2$. Therefore its trace is \bee
\mbox{tr}(K^2)=\int^1_{-1}k_2(x,x)\mathrm{d}x=-4,
\ee
and its Fredholm determinant is 
\be\label{eq5aa:sec4}
d(z)=\det(I_\mathcal{H}-z^2K^2)=\frac{1}{2}(\cosh(4z)+1).
\e
Note that $d_2(z)$ satisfy the conditions given in Theorem \ref{theom1a}. Hence we must have that 
\be\label{eq5a:sec4}
(d_2(z))^2=d(z).
\e
Indeed the above equality holds since 
\bee
(\cosh(2z))^2=\frac{1}{2}(\cosh(4z)+1).
\ee
Note that due to the discontinuity in the kernel function \eqref{eq4:sec4}, the approximate Hilbert--Schmidt determinant $d_{2N}(z)$ converges slowly to $d_2(z)$. However given the relation \eqref{eq5a:sec4} and that $k_2$ is continuous, one can improve the order of convergence in approximating $d_2(z)$ by approximating $\det(I_{\mathcal{H}}-z^2K^2)$ instead. Indeed, $k_2$ has similar properties as the kernel in Example \ref{exp:2}. Therefore, the same behaviour of the error can be expected as shown in Figure \ref{fig:2}.
\begin{rem}
 The same behaviour of the error holds for the self-adjoint operator $iK$ with eigenvalues $\pm i\la_n$.
\end{rem}

\end{exmp}
\begin{exmp}\label{exp:4}
 We consider equation \eqref{eq3a:sec4} with kernel given by
 \bee 
 k(x,y)=\frac{1}{\abs{x-y}^\alpha}.
 \ee
 The kernel function $k(x,y)$ is of the form given in \eqref{eq2aa:sec2} with $h(x,y)=1$. Hence the integral operator $K$ associated with the above kernel is compact, self-adjoint \cite{KA67} and positive. In particular, for $0<\alpha<1/2$ the integral operator $K$ is Hilbert--Schmidt (cf.\cite{HH73, CB77}). Therefore, the $2$-modified Fredholm determinant is computed as in Example \ref{exp:3} by substituting zero in the kernel function for $x=y$ \cite{DH04}. Then we shall focus on the case $\al=1/2$. For this case, the integral operator $K$ is not Hilbert--Schmidt since $\norm{k}_{L^2((-1,1)^2)}$ is unbounded hence $K\in\mathcal{J}_p(\mathcal{H})$ for $p\ge3$. Since $K$ is self-adjoint and positive operator so is the integral operator $K^3$. Moreover, the $3$-iterated kernel $k_3(x,y)$ is continuous (cf.\cite{CB77}) it then follows that the integral operator $K^3$ is of trace class (cf.~\cite[Theorem 2.12]{BS05}) with
 \bee
 \mbox{tr}(K^3)=\int^{1}_{-1}k_3(x,x)\mathrm{d}x=\norm{K^3}_1<\infty.
 \ee
 This implies since $K$ is self-adjoint and positive operator that
  \bee
 \mbox{tr}(\abs{K}^3)=(\mbox{tr}(K^3))^{1/3}<\infty.
 \ee
 Hence $K$ is in $\mathcal{J}_3(\mathcal{H})$. For the numerical computation of the $3$-modified Fredholm determinant $d_3(z)=\sideset{}{_3}\det(I_{\mathcal{H}}-zK)$, we need to compute numerically the eigenvalues of $K$ and form the finite dimensional version of equation \eqref{eq4a:sec1}. However, given Theorem \ref{theom1a} we are not required to, we only need to have an explicit expression of the $2$-iterated kernel $k_2(x,y)$ and set its diagonal values to zero. Using Maple, the $2$-iterated kernel $k_2(x,y)$ is
 \bee
 k_2(x,y)=
 \begin{cases}
 -\ln\left(2-y-x-2\sqrt{(1-y)(1-x)}\right)+\pi\\
\quad+\ln\left(2+y+x+2\sqrt{(1+y)(1+x) }\right),&x<y\\
-\ln\left(2+y+x-2\sqrt{(1+y)(1+x)}\right)+\pi\\
\quad+\ln\left(2-y-x+2\sqrt{(1-y)(1-x)}\right),&x>y
 \end{cases}
 \ee
 To confirm Theorem \ref{theom1a} for $K\in\mathcal{J}_3(\mathcal{H})$, we display in Figure \ref{fig:5} (left) the plot of $d_{3N}(z)=\sideset{}{_{3N}}\det(I_N-zK_N)$ computed using the five eigenvalues of largest modulus and $d_{2N}(z)=\sideset{}{_{2N}}\det(I_N-z^2K^2_N)$ computed by the rectangular rule. In Figure \ref{fig:5} (right), we plot $d_{3N}(z)d_{3N}(-z)$  and $d_{2N}(z)$.
\begin{figure}[!h]
\centering
\includegraphics[height=6cm, width=6cm]{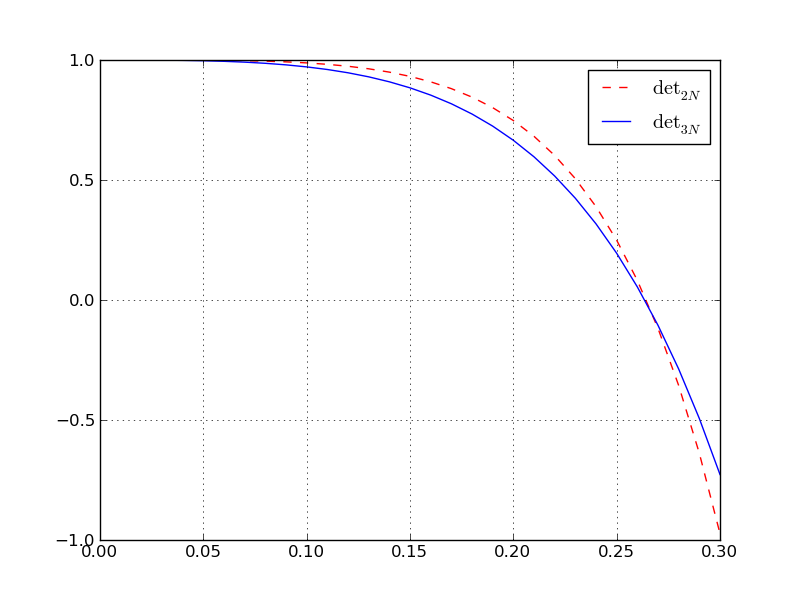}
\includegraphics[height=6cm, width=6cm]{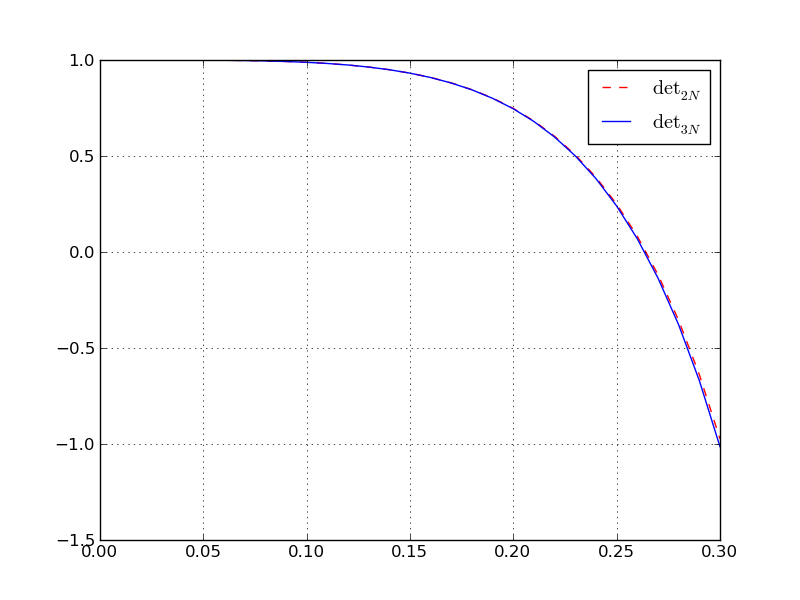}
\caption{ Numerical approximation of $d_3(z)$ (solid line) and $d_2(z)$ (left), and $d_3(z)d_3(-z)$ and $d_2(z)$ (right).}
\label{fig:5}
\end{figure}
Since $K$ is positive operator then its eigenvalues are positive. Therefore from Theorem \ref{theom1a}, we have for real $z\ge0$
 \bee
 d_{3N}(z)=c_N(z)d_{2N}(z),
 \ee
 where $c_N(z)=1/d_{3N}(-z)$ is a nonzero entire function. From Figure \ref{fig:5}, we observe that the smallest root $z_1=\la^{-1}_1$ of $d_{3N}(z)$ is simple, and together with the self-adjointness property of $K$ we conclude that eigenvalues of $K$ are simple. Moreover, we have observed numerically that the eigenfunctions of $K$ are real. Hence the error in evaluating the $3$-modified Fredholm determinant satisfies Corollary \ref{cor1}, that is for a fixed real $z>0$ we have
 \begin{align}
 \abs{d_{3}(z)-d_{3N}(z)}&=\abs{c(z)d_2(z)-c_{N}d_{2N}(z)}\label{eq6a:sec4}\\
&\le c\norm{Ku_1-K_Nu_1}\nonumber
\end{align}
 where $c(z)=1/d_3(-z)$. Since the error is proportionally the same then we can bound the above error with the error in approximating the $2$-iterated operator $K^2$, that is
 \begin{align*}
 \abs{d_{3}(z)-d_{3N}(z)}&\le \abs{c(z)}\abs{d_2(z)-d_{2N}(z)}\\
&\le c\abs{c(z)}\norm{K^2u_1-K^2_Nu_1}.
 \end{align*}
 It then follows from the uniform convergence of the Hilbert--Schmidt determinant $\sideset{}{_2}\det$ that equation \eqref{eq6a:sec4} convergences uniformly too. For numerical proof (cf.~Figure \ref{fig:6}), we plot the error $\abs{d_2(z)-d_{2N}(z)}$ using the rectangular rule evaluated at $z=0.1$ which according to the graph is of order $O(N^{-1}\log(N))$.
\begin{figure}[!h]
\centering
\includegraphics[height=6cm, width=6cm]{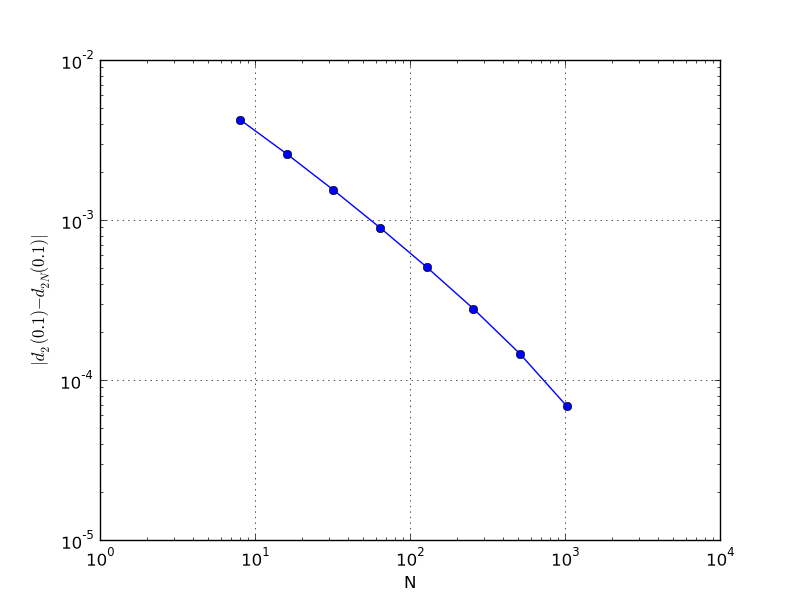}
\caption{ $\log_{10}$ of the error $\abs{d_2(z)-d_{2N}(z)}$ where $z=0.1$.}
\label{fig:6}
\end{figure}

 
\end{exmp}

\section{Conclusion}
In this paper, we have given theoretical and numerical results concerning the approximation error $(d_{p}-d_{Np})$, where $d_p$ is the $p$-modified Fredholm determinants and $d_{Np}$ is the finite dimensional determinants corresponding to $d_p$. These results are as follows; first, we have shown that the approximation error $(d_{p}-d_{Np})$ is uniform in a bounded domain. Second, we have given the rate of convergence when evaluating $(d_{p}-d_{Np})$ at an eigenvalue or at a point of the resolvent set. As a consequence, we have observed that numerical evaluation of the $p$-modified Fredholm determinants is nothing else than an interpolation in which the interpolation points are the eigenvalues of the operator $K$. From the well known result of interpolation theory which states that if the interpolated function is continuous then the error is uniform in a bounded domain, this confirms the uniform convergence obtained in Theorem~\ref{theom1}. 
Although we dealt with a finite domain of $\R$, an extension of the present analysis to $X\subseteq\R^n$ is possible. This is of course under the hypothesis that the kernel $k(x,y)$ is such that the set $\{K_N\}_{N\ge1}$ is collectively compact. However, numerically we believe that for $L^2$-kernel function $k(x,y)$ on the real line $\R$ and defined as in \eqref{eq1:sec4}, the NCC method will not be effective. This is due to the unbounded nature of the kernel functions $k_1$ and $k_2$. 
\section*{Acknowledgements}
It is a pleasure to thank Dr S.J.A Malham for suggesting the topic of this paper, and for many helpful discussions on its content. I would also like to thank Dr L. Boulton for useful discussions. Finally, I Would like to extend all my gratitude to the Numerical Algorithms and Intelligent Software Centre funded by the UK EPSRC grant EP/G036136 and the Scottish Funding Council for supporting this work.

\end{document}